\documentclass[12pt]{amsart}
\usepackage[utf8]{inputenc}
\usepackage{amssymb}
\usepackage{amsmath}
\usepackage{hyperref}
\usepackage{a4wide}
\usepackage{stackengine}

\DeclareMathOperator\id{\mathrm{id}}

\DeclareMathOperator\Span{\mathrm{span}}
\DeclareMathOperator\tr{\mathrm{tr}}

\newcommand\liegr{\mathsf}
\newcommand{\G}{\mbox{${\liegr G}_2$}}
\newcommand{\SO}[1]{\mbox{${\liegr {SO}}(#1)$}}
\newcommand{\SOxO}[2]{\mbox{$\liegr{S(O}(#1)\times{\liegr O}(#2))$}}
\newcommand{\Sp}[1]{\mbox{${\liegr {Sp}}(#1)$}}

\newcommand{\Spin}[1]{\mbox{${\liegr {Spin}}(#1)$}}
\newcommand{\SU}[1]{\mbox{${\liegr{SU}}(#1)$}}
\newcommand{\SUxU}[2]{\mbox{${\liegr {S(U}}(#1)\times{\liegr U}(#2))$}}
\newcommand{\U}[1]{\mbox{${\liegr U}(#1)$}}
\renewcommand{\O}[1]{\mbox{${\liegr O}(#1)$}}

\newcommand{\g}[1]{\mbox{$\mathfrak{#1}$}}
\newcommand{\La}{\mbox{$\mathfrak a$}}

\newcommand{\Lg}{\mbox{$\mathfrak g$}}
\newcommand{\Lh}{\mbox{$\mathfrak h$}}
\newcommand{\Lk}{\mbox{$\mathfrak k$}}
\newcommand{\Lm}{\mbox{$\mathfrak m$}}

\newcommand{\Lp}{\mbox{$\mathfrak p$}}

\newcommand{\Ls}{\mbox{$\mathfrak s$}}

\newcommand{\Lv}{\mbox{$\mathfrak v$}}
\newcommand{\Lw}{\mbox{$\mathfrak w$}}
\newcommand{\Lz}{\mbox{$\mathfrak z$}}

\newcommand\fieldsetc{\mathbb}
\newcommand{\C}{\fieldsetc{C}}
\newcommand{\F}{\fieldsetc{F}}
\newcommand{\R}{\fieldsetc{R}}

\renewcommand{\H}{\fieldsetc{H}}

\newcommand{\rotmat}[1]
{
   \begin{array}{|cr|}
     \hline
     \cos(#1) & -\sin(#1) \\
     \sin(#1) & \cos(#1) \\
     \hline
   \end{array}
}\newcommand{\onemat}[1]
{
   \begin{array}{|c|}
     \hline
     1 \\
     \hline
   \end{array}
}

\newtheorem{thm}{Theorem}[section]

\newtheorem{lem}[thm]{Lemma}
\newtheorem{prop}[thm]{Proposition}

\theoremstyle{remark}
\newtheorem{rmk}[thm]{Remark}

\newcommand\str{\rule[-.55em]{0em}{1.7em}}

\title[Stiefel manifolds]{Totally geodesic submanifolds and polar actions\\ on Stiefel manifolds}
\author[C.~Gorodski]{Claudio Gorodski}
\author[A.~Kollross]{Andreas Kollross}
\author[A.~Rodríguez-Vázquez]{Alberto Rodríguez-Vázquez}
\thanks{The first author acknowledges partial financial support from CNPq (grant  304252/2021-2) and FAPESP (grant 2023/00312-4). The third author has been supported by the projects PID2022-138988NB-I00 funded by MICIU/AEI/10.13039/501100011033 and by ERDF (European Union); and ED431F 2020/04 (Xunta de Galicia, Spain), the FWO Postdoctoral grant with project number 1262324N, and by the Horizon Europe research and innovation programme under Marie Sklodowska Curie Actions with grant agreement 101149711 - HOLYFLOW}

\address{Instituto de Matem\'atica e Estat\'\i stica,
Universidade de S\~ao Paulo,
Rua do Mat\~ao, 1010,
S\~ao Paulo, SP 05508-090, Brazil}
\email{claudio.gorodski@usp.br}

\address{Institut f\"{u}r Geometrie und Topologie, Universit\"{a}t Stuttgart, Pfaffenwaldring 57, 70550 Stuttgart, Germany}
\email{kollross@mathematik.uni-stuttgart.de}
	
\address{Départment de Mathématique,
	Université Libre de Bruxelles, Boulevard du Triomphe, CP 218,
	B-1050 Bruxelles, Belgium}
\email{alberto.rodriguez.vazquez@ulb.be}
\begin{document}

\pagenumbering{arabic}
\pagestyle{plain}
\begin{abstract}
We classify totally geodesic submanifolds of the real Stiefel manifolds of orthogonal two-frames. We also classify polar actions on these Stiefel manifolds, specifically, we prove that the orbits of polar actions are lifts of polar actions on the corresponding Grassmannian.  
In the case of cohomogeneity one actions we are able to obtain a classification for all real, complex and quaternionic Stiefel manifolds of $k$-frames.
\end{abstract}
\maketitle

%%%%%%%%%%%%%%%%%%%%%%%%%%%%%%%%%%%%%%%%%%%%%%%%%%%%%%%%%%%%%%%%%%%%%%%%%%%%%%%%%%%%%%%%%%%%%%%%%%%%%%%%%%%%%%%%%%
\section{Introduction}
%%%%%%%%%%%%%%%%%%%%%%%%%%%%%%%%%%%%%%%%%%%%%%%%%%%%%%%%%%%%%%%%%%%%%%%%%%%%%%%%%%%%%%%%%%%%%%%%%%%%%%%%%%%%%%%%%%

A connected submanifold $\Sigma$ of a Riemannian manifold $M$ is called \emph{totally geodesic} if every geodesic of~$\Sigma$ is also a geodesic of~$M$. In the case that $M$ is a space form, the classification of totally geodesic submanifolds is a well-known result. In fact, in a space of constant curvature, any linear subspace of the tangent space is tangent to a totally geodesic submanifold.
However, generic Riemannian manifolds do not have any totally geodesic proper submanifolds besides geodesics~\cite{murphy2019random}.
Between these two extremes: manifolds of the highest symmetry and manifolds without any symmetries, there lies the realm of Riemannian manifolds where we can expect a rich and interesting theory of totally geodesic submanifolds. For example, in Riemannian symmetric spaces, there are many different types of totally geodesic submanifolds.

Traditionally, totally geodesic submanifolds have been mostly studied in case $M$~is a Riemannian symmetric space. In such spaces, totally geodesic submanifolds can be described using the seemingly simple algebraic characterization by Lie triple systems. However the classification problem for totally geodesic submanifolds in symmetric spaces has only been solved in a few special cases, such as symmetric spaces of rank one~\cite{wolf1963elliptic}, and two~\cite{by1977totally,by1978totally,klein2008totally,klein2009totallyRS,klein2009totally,klein2010totally}, and products of rank-one symmetric spaces~\cite{R}. Under additional assumptions it has been solved for other symmetric spaces, such as exceptional symmetric spaces~\cite{kollross2023totally} (for maximal totally geodesic submanifolds), reflective submanifolds~\cite{leung,leung1979reflective,leung1979reflective4} (fixed point components of involutive isometries) and
non-semisimple maximal totally geodesic submanifolds~\cite{berndt2016maximal}.

Going beyond symmetric spaces, there have been few works in the past dealing with classification results on totally geodesic submanifolds in Riemannian homogeneous spaces, exceptions being the doctoral thesis~\cite{Kanschik2003} and the articles~\cite{kim2020totally} and \cite{Murphy}.
It is proved in~\cite{kim2020totally} that totally geodesic submanifolds of Damek-Ricci spaces are either locally isometric to rank-one symmetric spaces or non-symmetric Damek-Ricci spaces of lower dimension. In~\cite{Kanschik2003} and~\cite{Murphy}, totally geodesic submanifolds of Sasakian manifolds are investigated.

For the case of naturally reductive homogeneous spaces, there is an algebraic characterization of totally geodesic submanifolds due to Tojo~\cite{tojo}. This is a generalization of the Lie triple system condition for symmetric spaces. However, in general, it is not easy to exploit this condition in order to obtain classifications of totally geodesic submanifolds.

More recently, there has been renewed interest in developing intuition and tools for understanding totally geodesic submanifolds in non-symmetric homogeneous Riemannian manifolds, such as the complex, quaternionic and octonionic Hopf-Berger spheres~\cite{OR}. It is very interesting that in this classification some unexpected examples appear which are not orbits of groups of isometries of the ambient space, see also the remarks in~\cite{Kanschik2003} on the homogeneity of totally geodesic submanifolds in $\varphi$-symmetric Sasakian manifolds. As we see it,
this suggests that certain classes of non-symmetric homogeneous Riemannian
manifolds may accommodate more examples than previously thought.

In the present paper we consider a very restricted class of homogeneous manifolds which are very close to symmetric spaces, namely real Stiefel manifolds of orthonormal $2$-frames, which fiber over the real Grassmann manifolds of
$2$-planes, and we give a complete classification of maximal totally geodesic submanifolds; 
it turns out that all the examples are obtained from the Grassmannian, but our 
argument does not depend on the classification of
totally geodesic submanifolds of the Grassmannian. 
As mentioned above, some related previous work has been done in~\cite{Kanschik2003}
and~\cite{Murphy}. In the thesis~\cite{Kanschik2003}
there is a structural result on totally geodesic submanifolds of
$\varphi$-symmetric Sasakian manifolds (which include the real Stiefel
manifold of $2$-orthonormal frames), but it is not explicit and
assumes irreducibility of the submanifold. Also~\cite{Murphy}
studies totally geodesic submanifolds of Sasakian manifolds, but under
the assumption that the submanifold is $\varphi$-invariant,
$\varphi$-anti-invariant, or CR. 

Now we will state our main results. 
Let $M=V_2(\R^{n+2})$ be the real Stiefel manifold of pairs of orthonormal vectors in~$\R^{n+2}$, where $n\geq1$. 
Then $M = \O{n+2}/\O{n}=\SO{n+2}/\SO{n}$ is a homogeneous space and we will consider the normal homogeneous Riemannian metric induced by
\begin{equation}\label{eq:metric}
-\frac12\tr(XY), \quad \hbox{where $X,Y \in \g{so}_{n+2}$.}
\end{equation}
There is a homogeneous Riemannian submersion
\begin{equation}\label{fibration}
    F \longrightarrow M \stackon{$\longrightarrow$}{$\pi$} B
    \end{equation}
with totally geodesic fibers, where the base $B=G_2 (\R^{n+2})$ is the Grassmann manifold of $2$-planes in~$\R^{n+2}$ and the fiber $F$ is a circle.
As is well-known, the Stiefel manifold $V_2(\R^{n+2})$ is diffeomorphic to the unit tangent bundle of~$S^{n+1}$. In particular, $V_2(\R^4)$ is diffeomorphic to $S^3\times S^2$ and $V_2(\R^8)$ is diffeomorphic to $S^7\times S^6$, but the metric induced by~(\ref{eq:metric}) does not correspond to a product metric on the corresponding spheres, see~Remark~\ref{rmk:irreducible}. On the other hand, $V_2(\R^3)=\SO 3$ has constant curvature, so we need not consider it.  We show:

\begin{thm}\label{tg}
A connected complete maximal totally geodesic submanifold of the
Stiefel manifold $V_2(\R^{n+2})$, $n\geq2$, equipped with the metric induced by~(\ref{eq:metric}) is congruent to one of the following:
\begin{enumerate}
  \item a lower-dimensional Stiefel manifold~$V_2(\R^{n+1})$;
  \item the product of round spheres $S^p(1)\times S^q(1)$ with~$p+q=n$;
  \item the complex Hopf-Berger sphere~$S^{n+1}_{\mathbb C,\frac12}(\sqrt2)$ in case $n$~is even.
\end{enumerate}
In particular, all these submanifolds are reflective and thus closed.
Conversely, the above submanifolds are maximal totally geodesic submanifolds.
\end{thm}
In Subsection~\ref{refl} the terminology and explicit descriptions are given. The item \textup{(iii)} in  the above theorem corresponds to a certain complex Hopf-Berger sphere. These are constructed by endowing ${S}^{n+1}$ with the Berger metric $g_\tau$, which is obtained by starting with the standard round metric (with radius 1) and rescaling the vertical subspace of the Hopf fibration ${S}^1 \rightarrow {S}^{n+1} \rightarrow \C{\mathrm P}^\frac{n}{2}$ by a factor $\tau > 0$. We denote by ${S}^{n+1}_{\C,\tau}(r)$ the sphere ${S}^{n+1}$ equipped with the Riemannian metric $r^2 g_\tau$. Using Theorem~\ref{tg} we can recover all totally geodesic submanifolds by proceeding inductively (for products of rank-one symmetric spaces see~\cite{R}, and for the complex Hopf-Berger sphere see~\cite{OR}).

Every complete totally geodesic submanifold of a homogeneous 
Riemannian manifold is intrinsically homogeneous 
(cf.~\cite[Ch.~VII, Cor.~8.10]{kn}), 
but in case the ambient manifold is not a symmetric space
there is no a priori reason for it to be extrinsically homogeneous. 
Recall that a Riemannian homogeneous space is called a \emph{g.o.~space} if each of its geodesics is an orbit of a one-parameter group of isometries. All naturally reductive homogeneous spaces are g.o.~spaces, so in particular $V_2(\R^{n+2})$ is a g.o.~space. 
However, we show that $V_2(\R^{n+2})$, $n\geq2$, contains a unique congruence class of a totally geodesic submanifold isometric to $\R {\mathrm P}^2$ which is not extrinsically homogeneous, see~Proposition~\ref{prop:noho}.
To our knowledge, this is the first example of a not extrinsically homogeneous totally geodesic submanifold in a normal homogeneous space.

\bigskip

In the second part of this paper, we deal with polar actions on $V_2(\R^{n+2})$. Recall that a polar action on a complete Riemannian manifold is a proper and isometric action admitting a section, i.e.\ a connected complete submanifold which meets all orbits and meets them always orthogonally.

The classification of polar actions on symmetric spaces is a research topic whose history goes back several decades. 
It is closely related to the study of totally geodesic submanifolds, since sections are always totally geodesic.  In particular, polar actions were classified in compact rank one symmetric spaces, see~\cite{PTJDG} and~\cite{GK}. Biliotti conjectured that any polar action on an irreducible compact symmetric space of rank higher than one is hyperpolar, i.e.\ has flat sections. The conjecture was confirmed in a series of papers~\cite{kpolar,kollrossexcep,KL} and, since hyperpolar actions were classified in~\cite{kollross}, this led to the classification of polar actions on irreducible compact symmetric spaces. However, it remains an open question whether Billioti's conjecture can be extended to a broader context.  We we will see that this is the case for $V_2(\R^{n+2})$.

Many examples of polar actions are given by cohomogeneity one actions. We obtain a classification of all cohomogeneity one actions and thus of all (extrinsically) homogeneous hypersurfaces of all real, complex and quaternionic Stiefel manifolds~$V_k(\F^n) = G/H $, where $G = \SO n$, $\U n$, or $\Sp n$, respectively, and $H$ is the subgroup of~$G$ that fixes the first~$k$ canonical basis vectors~$e_1, \dots, e_k$ of~$\F^n$.

\begin{table}[h]
	\centering
	\begin{tabular}{|c|c|}
		\hline
		\str $L \times K$ & $M$ \\ \hline \hline
		\str $\SO{n-1} \times \SO{k}$ & $V_k(\R^n)$ \\ \hline
		\str $\Spin9 \times \SO2$ & $V_2(\R^{16})$ \\ \hline
		\str $\Sp{m}\Sp1 \times \SO2$ & $V_2(\R^{4m})$ \\ \hline
		\str $\U{m} \times \SO2$ & $V_2(\R^{2m})$ \\ \hline
		\str $\U{m} \times \SO3$ & $V_3(\R^{2m})$ \\ \hline
		\str $\G \times \SO{3}$ & $V_3(\R^7)$ \\ \hline\hline
		\str $\Sp{m} \times \U2$ & $V_2(\C^{2m})$ \\ \hline
		\str $\Sp{m} \times \U3$ & $V_3(\C^{2m})$ \\ \hline
		\str $\SUxU{n-1}{1} \times \U{k}$ & $V_k(\C^n)$ \\ \hline\hline
		\str $\left[ \Sp{n-1}\times\Sp{1} \right] \times \Sp{k}$ & $V_k(\H^n)$ \\ \hline
	\end{tabular}
	\caption{ \str Isometric actions on $M$ of cohomogeneity one.}\label{t:hhyp}
\end{table}

\begin{thm}\label{thm:hhls}
Let $M = G/H = V_k(\F^n)$, $k \in \{ 2, \dots, n-1\}$, be a Stiefel manifold equipped with a normal homogeneous metric.
Let $S$ be an (extrinsically) homogeneous hypersurface of $M$. 
Then the following holds.
	\begin{enumerate}
		\item There is a homogeneous hypersurface $U$ of the Grassmannian $G_k(\F^n)$ such that $S = \pi^{-1} (U)$.
		\item The hypersurface $S$ is congruent to a regular orbit of one of the actions of a Lie group~$L \times K$ on the Stiefel manifold~$M = G/H$, where $L \times K$ and~$M$ are as in~Table~\ref{t:hhyp}, and where the action is given by
		\begin{equation}\label{eq:stac}
		(\ell,k) \cdot gH = \ell g k H\quad\mbox{for$\quad\ell \in L,\; k \in K,\; g \in G$.}
		\end{equation}	
	\end{enumerate}
\end{thm}
\bigskip

Finally, we show that the sections of polar actions on~$V_2(\R^n)$ are horizontal with respect to the fibration $V_2(\R^n) \to G_2(\R^n)$, which implies that they all come from polar actions on the corresponding Grassmannian.
Our method here works for the more general class of $G$-invariant homogeneous metrics.

\medskip
\begin{thm}\label{pa}
Let $M=V_2(\R^{n+2})$, $n \ge 2$, be equipped with an $\SO{n+2}$-invariant Riemannian metric.
Suppose a compact connected Lie group acts polarly on~$M$. Then there exists a polar action on~$B=G_2(\R^{n+2})$ such that the orbits of the action on~$M$ coincide with the preimages of the orbits of the action on~$B$ under the canonical projection $M \to B$. In particular, every polar action on~$M$ is orbit equivalent to either one of the following:
\begin{enumerate}
\item the action of a subgroup~$L \times K \subset \SO{n+2} \times \SO2$ acting with cohomogeneity one, as given in the first four rows of~Table~\ref{t:hhyp}, or;
\item the action of a subgroup
\[
L \times K = [\SO q \times \SO{n+2-q}] \times \SO2, \quad q, n+2-q \ge 2,
\]
of $\SO{n+2} \times \SO2$, which acts hyperpolarly and with cohomogeneity two.
\end{enumerate}
In both cases, the action of $L \times K$ on~$M$ is defined by~(\ref{eq:stac}).
\end{thm}
The restriction~$n \ge 2$ is necessary because $V_2(\R^3)$ is isometric to $\R {\mathrm P}^3$ and there is an action of~$\SO2$ on~$\R {\mathrm P}^3$ that has a non-flat section isometric to $\R {\mathrm P}^2$.

It is a natural question to ask whether Theorems~\ref{tg} and~\ref{pa} can be generalized to Stiefel manifolds of $k$-frames with $k\ge 3$ over $\R$, $\C$, or $\H$.
We expect these to be harder problems because for the classification of totally geodesic submanifolds 
in case $k=2$ we have strongly used the fact that the fiber of the fibration $V_2(\R^{n+2}) \to G_2(\R^{n+2})$ is one-dimensional.
In the case of polar actions, this is also used in Lemma~\ref{lm:toru} to show that sections must be horizontal.
%%%%%%%%%%%%%%%%%%%%%%%%%%%%%%%%%%%%%%%%%%%%%%%
%%%%%%%%%%%%%%%%%%%%%%%%%%%%%%%%%%%%%%%%%%%%%%%%

\subsection*{Acknowledgements} 
We are grateful to Juan Manuel Lorenzo-Naveiro for useful conversations on a preliminary version of this article.

%%%%%%%%%%%%%%%%%%%%%%%%%%%%%%%%%%%%%%%%%%%%%%%%%%%%%%%%%%%%%%%%%%%%%%%%%%%%%%%%%%%%%%%%%%%%%%%%%%%%%%%%%%%%%%%%%%
\section{Preliminaries}
%%%%%%%%%%%%%%%%%%%%%%%%%%%%%%%%%%%%%%%%%%%%%%%%%%%%%%%%%%%%%%%%%%%%%%%%%%%%%%%%%%%%%%%%%%%%%%%%%%%%%%%%%%%%%%%%%%

\subsection{Reductive decompositions}\label{subs:redec}

We refer the reader to~\cite{gw} for the relevant facts about Riemannian
submersions. We write the Riemannian submersion~(\ref{fibration}) as
\[ K/H \longrightarrow G/H \stackon{$\longrightarrow$}{$\pi$} G/K. \]
Associated with the symmetric space $B=G/K = \SO{n+2}/\SOxO{2}{n}$ is the decomposition $\Lg = \Lk +\Lp$, where $\Lg=\mathfrak{so}_{n+2}$, $\Lk=\mathfrak{so}_2\oplus\mathfrak{so}_n$ and $\Lp=\R^2\otimes\R^n$ as a $K$-module; 
Consider the reductive decomposition~$\Lg=\Lh+\Lm$ of~$M$, where $\Lh=\mathfrak{so_n}$ and $\Lm=\Lz + \Lp$, with $\dim\Lz=1$  and where $o = eH$~denotes the base point. 
The space~$\Lm$ is identified with the tangent space~~$T_oM$
and $\Lp$~is identified with~$T_{\pi(o)}B$. 
We identify $\Lg$~with the skew-symmetric real matrices of order~$n+2$, and the subalgebra~$\Lk$ with the block diagonal matrices as follows,
\[ \Lk= \left\{\left( \begin{array}{cc} P & 0 \\ 0 & Q \end{array}\right)\colon
\ P\in\mathfrak{so}_2, Q\in\mathfrak{so}_n\right\}, \]
so that
\[ \Lp = \left\{ \left( \begin{array}{cc} 0 & Z \\ -Z^t & 0 \end{array}\right)
\colon\ Z\in M_{2\times n}(\R) \right\}. \]

Recall that $B$ is a globally Riemannian symmetric space covered by a Hermitian symmetric space, and there is $V\in\Lz$ such that $J=-\mathrm{ad}_V|_{\mathfrak p}$ defines a complex structure on~$\Lp$. For $1 \le i , j \le n+2$ we define~$E_{ij}$ to be the real $(n+2) \times (n+2)$-matrix which has $1$ in the $(i,j)$-entry and $0$~elsewhere. For $i < j$ we define $X_{ij}=E_{ij}-E_{ji}$. We can take $V=X_{12}$.

Let us consider the symmetric bilinear map $U\colon \g{m} \times \g{m}\rightarrow \g{m}$ defined by
\[2\langle U(X,Y), Z\rangle = \langle [Z, X]_{\g{m}}, Y \rangle + \langle X, [Z, Y ]_{\g{m}}\rangle\] where $X, Y, Z \in \g{m}$, $\langle {\cdot},{\cdot} \rangle$ is an $\mathrm{Ad}_G(H)$-invariant scalar product on~$\g{m}$ and $(\cdot)_{\g{m}}$ denotes the natural projection onto~$\g{m}$. The reductive decomposition $\g{g}=\g{h}+\g{m}$ is said \textit{naturally reductive} if $U$ is identically zero.

Finally, it is important to notice that the identity component of the isometry group of~$M$ contains $\SO{n+2} \cdot \SO 2$, where the $\SO 2$-factor acts freely from the right on~$M=G/H$; in fact, it normalizes $H$.
It will be shown in Lemma~\ref{lemma:stis} below that the identity component of~$\mathrm{Isom}(M)$ is precisely $\SO{n+2}\cdot \SO 2$.

\subsection{Totally geodesic submanifolds}

In our case, $M$~is a homogeneous Riemannian manifold, and hence analytic and complete. Hence every locally defined totally geodesic submanifold of~$M$ can be uniquely extended to a complete totally geodesic submanifold (possibly with self-intersections) by exponentiating the tangent space at any of its points~\cite{h}. By homogeneity, it is also sufficient to classify totally geodesic submanifolds of~$M$ passing through~$o$.  It is convenient to call a subspace of~$\Lm$ a \emph{totally geodesic subspace} if it exponentiates to a
totally geodesic submanifold of~$M$.

The class of naturally reductive homogeneous spaces~$M=G/H$ includes the normal homogeneous spaces. For naturally reductive homogeneous spaces with reductive decomposition $\Lg=\Lh+\Lm$, we recall the following result from~\cite{tojo}, where $R$ denotes the curvature tensor of~$M$:
\begin{thm}[Tojo's criterion]\label{tojo}
A subspace $\Ls$ of~$\Lm$ is totally geodesic
if and only if
$e^{\nabla X}(\Ls)$ is $R$-invariant for all $X\in\Ls$.
\end{thm}

\subsection{Reflective submanifolds}\label{refl}

A \emph{reflective submanifold} of a Riemannian manifold is a connected component of the fixed point set of an involutive isometry of the manifold; it is automatically a totally geodesic submanifold. Reflective submanifolds of symmetric spaces were classified by Leung~\cite{leung}. In the case of~$B$, the associated isometries lift to involutive isometries of~$M$ and give rise to three families
of reflective submanifolds of~$M$, as follows.

\begin{enumerate}
\item Let $g=\left(\begin{smallmatrix}I_p&0\\0&-I_q\end{smallmatrix}\right)$,
where $p+q=n+2$.
Then the fixed point set is $M^g=V_2(\R^{p})$.
\item Let $g=\left(\begin{smallmatrix}-1&0&0&0\\0&1&0&0\\
0&0&I_p&0\\0&0&0&-I_q\end{smallmatrix}\right)$, where $p+q=n$.
Then $M^g=S^p(1)\times S^q(1)$ (a Riemannian product of round spheres of radius~$1$).
As a special case, $q=0$ gives $M^g=V_1(\R^{n+1})=S^n(1)$.
\item If $n$ is even, take $\tau$ to be conjugation by the matrix
$\left(\begin{smallmatrix}J_1&0\\0&J_{\frac n2}\end{smallmatrix}\right)$,
where $J_k=\left(\begin{smallmatrix}0&-I_k\\I_k&0\end{smallmatrix}\right)$.
Then $\tau$ is an involutive automorphism of~$\Lg=\mathfrak{so}_{n+2}$  which
preserves the decomposition $\Lh+\Lm$ and hence induces an involutive
isometry of~$M$ whose fixed point set is the reflective manifold
$S^{n+1}_{\mathbb C,\frac12}=\SU{\frac n2+1}/\SU{\frac n2}\subset \SO{n+2}/\SO{n}=M$ (a complex Hopf-Berger
sphere).
\end{enumerate}

In Section~\ref{classif} we show that every maximal totally geodesic submanifold of~$M$ falls into one of these classes, thereby proving Theorem~\ref{tg}.

\subsection{Isotropy representation of the Grassmannian}\label{isotr}

The Grassmannian~$B=G_2(\R^{n+2})$ is a rank-two  symmetric space of compact type. For $n>2$, its isotropy representation is the (irreducible) representation of~$K=\SOxO{n}{2}$ on~$\Lp=\R^n\otimes\R^2$, which is a polar
representation. In fact, a maximal abelian subspace~$\La$ of~$\Lp$ is two-dimensional and meets all the orbits orthogonally, so it is a section.
The restricted root system is of type~$\mathrm{B}_2$, the restricted Weyl group~$\mathcal W$ is the dihedral group~$\mathrm{D}_4$ (with $8$~elements) and hence the orbit space $\Lp/K=\La/\mathcal W$ is the simplicial cone of angle~$\pi/4$ in the plane. Besides the origin, there are two singular orbit types, namely $\SOxO{n-1}1$ and
$\SOxO{n-2}2$.

It is convenient to choose $\La=\Span \{ X_{13},X_{24} \} $ and write representatives for the $K$-orbits of unit vectors in~$\Lp$ as $\cos (t) X_{13}+\sin (t) X_{24}$, where $t\in[0,\pi/4]$, so that $t=0$ and $t=\pi/4$ parametrize the
singular orbits and $t\in(0,\pi/4)$ parametrize the principal orbits.

\subsection{The curvature}

Since $\pi:M\to B$ is a Riemannian submersion with totally geodesic $1$-dimensional fibers, we can use O'Neill's formulas to derive the expression for the curvature tensor of~$M$ in terms of the curvature tensor of the symmetric space~$B$ and O'Neill's tensor~$A$. We write $\mathcal H$ and~$\mathcal V$ for the sections of the horizontal and vertical subbundles of the tangent bundle. Then $T_oM=\Lm$, $\mathcal H_o=\Lp$ and $\mathcal V_o=\Lz$, and the $A$-tensor defined by
\[
A \colon \mathcal H\times\mathcal H \to \mathcal V,\ A_XY = (\nabla_XY)^v=\frac12[X,Y]^v \]
is given by $A_XY=\frac12\langle X,JY\rangle V$. Therefore the dual tensor $A^* \colon \mathcal H\times\mathcal V\to\mathcal H$ is $A^*_XV=-\frac12JX$. A simple calculation also shows that $(\nabla_E^v A)_XY=0$
for $X$, $Y\in\mathcal H$ and $E\in TM$, where $\nabla^v$ denotes the
vertical component of the Levi-Civita connection of $M$. 

The curvature tensor of the symmetric space $B$ is given by
\[ R^B(X,Y)Z= - [[X,Y],Z] \]
for $X$, $Y$, $Z\in\Lp$~\cite[Theorem~8.4.1]{w}.
Regarding the curvature tensor $R$ of $M$,
for now it suffices to consider the Jacobi operators
\[ R_X=R(\cdot,X)X \colon \Lm\to\Lm\quad\mbox{and}\quad R_V=R(\cdot,V)V \colon \Lp\to\Lp \]
for $X\in\Lp$. Now the first and third curvature identities
in~\cite[p.~44]{gw} yield
\begin{equation}\label{r-h}
 R_X(Y)=-\mathrm{ad}_X^2Y-\frac34\langle Y,JX\rangle JX\quad
\mbox{($Y\in\Lp$)},
\end{equation}
\[ R_V|_{\mathfrak p}=\frac14 \id_{\mathfrak p},  \]
and therefore
\begin{equation}
\label{r-hv} R_X(V) = \frac14 V.
\end{equation}

\subsection{Restricted roots}\label{roots}

Fix a maximal abelian subspace $\La$ of~$\Lp$ and suppose $X\in\La$. Recall the (orthogonal) restricted root decomposition
\[
\Lp=\La+\Lp_{\theta_1}+ \Lp_{\theta_2}+ \Lp_{\theta_1+\theta_2}+ \Lp_{\theta_1-\theta_2},
\]
where the dimensions are respectively $2$, $n-2$, $n-2$, $1$ and~$1$; where $\theta_1$, $\theta_2$ is the dual basis to $X_{13}$, $X_{24}$; we have chosen the scaling of the metric on~$M$ to make it an orthonormal basis (in particular $V=X_{12}$ is a unit vector). It is very easy to see that $\La$ is totally real in~$\Lp$, see~\cite[Lemma~2.7]{pt} and, indeed, $J\La=\Lp_{\theta_1+\theta_2}+\Lp_{\theta_1-\theta_2}$, where $\Lp_{\theta_1\pm\theta_2}=\Span \{ X_{23}\mp X_{14} \} $, see~\cite[p.~109]{l}. In addition, the second term in~(\ref{r-h}) vanishes if $Y\perp J\La$. It follows that $R_X(\La)=0$ and $R_X$ on~$\Lp_{\theta_i}$ is given by~$-\mathrm{ad}_X^2=\theta_i(X)^2\id$ for~$i\in\{1,2\}$. 

\subsection{Isometric actions}

We refer the reader to~\cite{bco,palais2006critical} for results about isometric actions and
polar actions. 
An isometric action of a compact Lie group~$H$ on a complete connected Riemannian manifold~$M$ is called \emph{polar} if there exists a \emph{section}, i.e.~a closed submanifold~$\Sigma$ such that $H \cdot \Sigma = M$ and $T_p(H \cdot p) \perp T_p\Sigma$ for each point $p \in \Sigma$. Sections are always totally geodesic submanifolds~\cite[Theorem 5.6.7]{palais2006critical} and any two sections of a polar action are conjugate by an element of the group~$H$. If sections are flat in the induced Riemannian metric, then the action is called \emph{hyperpolar}. Note that a compact Lie group~$H$ acts polarly on a Riemannian manifold if and only if the action restricted to the connected component of~$H$ is polar.
Recall that the \emph{cohomogeneity} of a Lie group action is defined as the minimal codimension of its orbits.

%%%%%%%%%%%%%%%%%%%%%%%%%%%%%%%%%%%%%%%%%%%%%%%%%%%%%%%%%%%%%%%%%%%%%%%%%%%%%%%%%%%%%%%%%%%%%%%%%%%%%%%%%%%%%%%%%%
\section{The classification of totally geodesic
submanifolds}\label{classif}
%%%%%%%%%%%%%%%%%%%%%%%%%%%%%%%%%%%%%%%%%%%%%%%%%%%%%%%%%%%%%%%%%%%%%%%%%%%%%%%%%%%%%%%%%%%%%%%%%%%%%%%%%%%%%%%%%%

\subsection{Every totally geodesic subspace contains a singular point}\label{sec:sing}

Let $\Sigma$ be a totally geodesic submanifold of~$M$ with~$\dim\Sigma\geq2$. For dimensional reasons, the tangent spaces to $\Sigma$ have a non-trivial intersection with the horizontal distribution of~$\pi \colon M\to B$. In this section we prove that there exists $p\in\Sigma$ such that $T_p\Sigma\cap\mathcal H_p$ contains a non-zero vector which projects to a singular point for the isotropy representation of~$M$ at~$\pi(p)$. The main tool is Tojo's criterion, see Theorem~\ref{tojo}, which we use only up to first order. Even more basic, any Jacobi operator~$R_X$ for~$X\in\Ls$ leaves $\Ls$ invariant; since this is a symmetric operator, it is diagonalizable over~$\R$ and hence $\Ls$ must be spanned by certain eigenvectors of~$R_X$. Note that $R_V$ is zero on~$\Lz$ and scalar on~$\Lp$, so it gives no information on~$\Ls$ and
we will rather consider $R_X \colon \Lm\to\Lm$ for $X\in\Ls\cap\Lp$.

Using the isometry group $\SO{n+2} \times \SO 2$ of~$M$, we can assume that $o\in\Sigma$ and
\[
X= \cos (t) X_{13}+\sin (t) X_{24} \in\Ls\cap\La,
\]
where $\Ls=T_o\Sigma$ is a totally geodesic subspace of~$\Lm$ 
(cf.~Subsection~\ref{isotr}). The main result of this section is that we may further assume that $t=0$ or $t=\pi/4$. The proof is by contradiction: suppose on the contrary that $\Ls\cap\Lp$ does not contain singular vectors.
\begin{table}
\begin{tabular}{ccc}
eigenvalue & eigenspace & multiplicity \\
\hline
 $0$ & $\La$ & $2$ \\
$1/4$ & $\Lz$ & $1$ \\
$\lambda_\pm$ & $\R X_\pm$ & $1$ \\
$\cos^2t$ & $\Lp_{\theta_1}$ & $n-2$ \\
$\sin^2 t$ & $\Lp_{\theta_2}$ & $n-2$ \\
\end{tabular}
\caption{Eigenspace decomposition of~$R_X$ for $t\in(0,\pi/4]$.}\label{1}
\end{table}
It is easy to see that $JX$ has non-zero components in $\Lp_{\theta_1\pm\theta_2}$ if and only if $\theta_1(X)\mp\theta_2(X)\neq0$; it then follows from~(\ref{r-h}) that $JX$ is not an eigenvector of~$R_X$ for $t\in[0,\pi/4)$. In particular, $R_X$ has different eigenvalues on~$J\La$, say $\lambda_{\pm}$. It is easy to compute that the eigenvalues of~$R_X$ on~$J\La$ for $t\in[0,\pi/4]$ are \[ \lambda_\pm = \frac{5\pm\sqrt{9+16\sin^2 2t}}8.  \] Note that $\lambda_+$ increases from $1$ to~$\frac54$, and that $\lambda_-$ decreases from $\frac14$ to~$0$ for $t\in[0,\pi/4]$.

Table~\ref{1} gives the complete eigenspace decomposition of~$R_X$ for $t\in(0,\pi/4]$; there 
\[ \Lp_{\theta_1}=\Span\{X_{15},\ldots,X_{1,n+2}\}, \qquad \Lp_{\theta_2}=\Span\{X_{25},\ldots,X_{2,n+2}\}. \] 
We do not need the precise expressions for~$X_\pm$, except that in case $t=\pi/4$ we have $X_+=\frac1{\sqrt2}(X_{23}-X_{14})=JX$ and
$X_-=\frac1{\sqrt2}(X_{23}+X_{14})$.

\begin{lem}\label{linalg}
Let $W$ be a finite dimensional real inner product space, let $S$, $T \colon W\to W$ be endomorphisms and assume that $T$ is symmetric. Let $U$ be a $T$-invariant subspace and consider the $1$-parameter family of automorphisms  $\{\varphi_t=e^{tS}\}$ of~$W$. Suppose $U_t=\varphi_tU$ is $T$-invariant for all~$t$. Then $U$ is $[S,\pi_\lambda]$-invariant for all eigenvalues $\lambda$ of~$T$, where $\pi_\lambda$ denotes the orthogonal projection of~$W$ onto the $\lambda$-eigenspace of~$T$ in~$W$, and $[\cdot,\cdot]$ the commutator of linear maps.
\end{lem}

\begin{proof}
Let $u\in U$. Then $\varphi_t(u)\in U_t$. Since $U_t$ is $T$-invariant and $T$ is symmetric, the component $\pi_\lambda\varphi_t(u)\in U_t$, hence $\varphi^{-1}_t\pi_\lambda\varphi_t(u)\in U$ for all $t$. The desired condition is now obtained by differentiating with respect to~$t$ at~$0$.
\end{proof}

\begin{lem}\label{linalg2}
Under the assumptions of Lemma~\ref{linalg}, write $W=\oplus W_\lambda$ for the eigenspace decomposition of~$T$. If $u\in W_\lambda\cap U$ then $\pi_\mu(Su)\in U$ for all $\mu\neq\lambda$.
\end{lem}

\begin{proof}
By Lemma~\ref{linalg}, $\pi_\mu(Su) = \pi_\mu(Su) -S \pi_\mu(u) = [\pi_\mu,S]u\in U$ as desired.
\end{proof}

\begin{lem}\label{a12}
$\Ls\cap\La=\R X$ and $\Ls\cap\Lp_{\theta_1}=\Ls\cap\Lp_{\theta_2}=\{0\}$.
\end{lem}

\begin{proof}
In fact, if $\La\subset \Ls$ or  $\Ls\cap\Lp_{\theta_1}\neq\{0\}$ or $\Ls\cap\Lp_{\theta_2}\neq\{0\}$, then $\Ls$ would contain singular vectors, contrary to our assumption.
\end{proof}

\begin{lem}\label{vxpm}
The following statements are equivalent:
\begin{enumerate}
  \item $X_+\in\Ls$ or $X_-\in\Ls$;
  \item $V\in\Ls$;
  \item $X_+\in\Ls$ and $X_-\in\Ls$.
\end{enumerate}
\end{lem}

\begin{proof}
Take $S=-\nabla X^*$, $U=\Ls$, $T=(R_X)_o$ and $W=\Lm$ in Lemma~\ref{linalg2}, where $X^*$ is the Killing field on~$M$ induced by $X\in\Lg$ and $\nabla$ denotes the Levi-Civita connection on~$M$. Due to Tojo's criterion~\ref{tojo}, $U_t$ is $T$-invariant for all $t$. Since $M$ is normal homogeneous, $(\nabla_YX^*)_o=\frac12[X,Y]_{\mathfrak m}$ for $Y\in\Lm$, where $\Lm$ is identified with $T_oM$.
Now
\begin{equation}\label{SV}
 SV=-\frac12JX,
\end{equation}
and
\begin{equation}\label{SXpm}
 \langle SX_\pm,V\rangle=-\frac12\langle [X,X_{\pm}],V\rangle =
\frac12\langle JX,X_{\pm}\rangle.
\end{equation}
For $t\in(0,\pi/4)$, $JX$ is not an eigenvector of~$R_X$ and therefore it is a linear combination of~$X_+$, $X_-$ with non-zero coefficients.

If $X_\pm\in\Ls$, then $X_\pm\in\Ls\cap\Lm_{\lambda_\pm}$, so $\pi_{1/4}(SX_\pm)\in\Ls$ by Lemma~\ref{linalg2}, and it is a non-zero multiple of~$V$ by~(\ref{SXpm}), which shows that $V\in\Ls$.

If $V\in\Ls$, then $V\in\Ls\cap\Lm_{1/4}$, so $\pi_{\lambda_\pm}(SV)\in\Ls$ by Lemma~\ref{linalg2}, and it is a non-zero multiple of~$X_\pm$ by~(\ref{SV}), showing that $X_\pm\in\Ls$.
\end{proof}

It follows from Lemma~\ref{vxpm} that $V$, $X_{\pm}\not\in\Ls$ (for otherwise $\Ls$ would contain singular vectors, since a certain linear combination of~$X_+$, $X_-$ is a singular vector). Together with Lemma~\ref{a12} this shows that $\Ls=\R X$, which contradicts the assumption that $\dim\Ls\geq2$. This proves that $\Ls$~contains a singular vector.

\subsection{The classification}

Let $\Ls$ be a totally geodesic subspace of~$\Lm$.

\begin{lem}\label{3}
If $Y$, $JY\in\Ls\cap\Lp$ and $[Y,JY]_{\mathfrak h}$ centralizes $\Span\{ Y,JY \}^\perp\cap\Lp\cap\Ls$, then $\Ls\cap\Lp$ is a complex subspace of~$\Lp$.
\end{lem}

\begin{proof}
We will use the formula
\begin{equation}\label{f}
 R(X,Y)Z=-[[X,Y],Z] -\frac12\langle X,JY\rangle JZ+\frac14\langle Y,JZ\rangle
JX+\frac14\langle Z,JX\rangle JY
\end{equation}
for $X$, $Y$, $Z\in\Lp$,
which is obtained from~\cite[p.~44]{gw}.

Let $Y$, $JY$ be an orthonormal pair in~$\Ls\cap\Lp$ as in the statement. Given $Z\in\Ls\cap\Lp$, we want to show that $JZ\in\Ls\cap\Lp$. We may assume that $Z\perp Y$, $JY$. Now~(\ref{f}) gives \[ R(Y,JY)Z=-[[Y,JY],Z] +\frac12 JZ. \] By our assumption, \[ [[Y,JY],Z]=[[Y,JY]_{\mathfrak z},Z]=[-V,Z]=JZ, \] so $R(Y,JY)Z =  -\frac12JZ$. Since $\Ls$ is $R$-invariant, this implies that $JZ\in\Ls\cap\Lp$, as desired.
\end{proof}

We proceed with the classification. According to Section~\ref{sec:sing}, we may assume $X_{13}\in\Ls$ or $X_{13}+X_{24}\in\Ls$.

\subsection{The case $X=X_{13}\in\Ls$}\label{sec:x13}

We need the eigenspace decomposition of~$R_X$, which is given
by Table~\ref{2}.
\begin{table}
\begin{tabular}{ccc}
eigenvalue & eigenspace & multiplicity \\
\hline
 $0$ & $\La+\Lp_{\theta_2}$ & $n$ \\
$1/4$ & $\Span\{X_{12}, X_{23}\}$ & $2$ \\
$1$ & $\R X_{14}+\Lp_{\theta_1}$ & $n-1$ \\
\end{tabular}
\caption{Eigenspace decomposition of~$R_X$ for $X=X_{13}$.}\label{2}
\end{table}
We divide the discussion into two cases.

\subsubsection{Complex case.}

Assume $\Ls$ has a component in the $1/4$-eigenspace $\Span \{ V,JX \}$.  Let $0\neq Y =aV+bJX\in\Ls$ for some $a$, $b\in\R$. Note that
\[
\nabla_YX^* = \frac12[X,Y]_{\mathfrak m}=\frac12(-bV+aJX).
\]
Let $\tau_t$ be the parallel displacement along the geodesic $\gamma(t)=\exp(tX)$ from $0$ to $t$. We use the formula~\cite[Lemma~3.1]{tojo} \begin{equation}\label{tau} d(L_{\exp(tX)})^{-1}\tau_t(Y)=e^{-t\nabla X}\cdot Y, \end{equation} which is clearly a rotation in the $\Span\{V,JX\}$-plane. Since $\Sigma$ is invariant under $\tau_t$, by changing the base point we may assume that $X$, $JX\in\Ls$. Since the remainder of~$\Ls\cap\Lp$ is spanned by $1$- and $0$-eigenvectors of~$R_X$, Lemma~\ref{3} implies that, up to the action of $\SO{n-1}\subset H$, we have
\[
\Ls\cap\Lp=\Span\{X_{13},\ldots,X_{1k},X_{23},\ldots,X_{2k}\}
\]
for some $3\leq k\leq n+2$.

It remains to be decided whether $V\in\Ls$. Note that $\Sigma$ can be regarded as a totally geodesic submanifold of $V_2(\R^k)\subset M$ of codimension at most $1$. Thanks to the main theorem in~\cite{tsukada}, we know that a simply connected irreducible naturally reductive homogeneous space can admit a totally geodesic hypersurface if and only if it has constant curvature. Since $M=V_2(\R^{n+2})$ is simply connected for $n\geq2$ (as can be seen from the long exact sequence in homotopy of the fibration $\SO{n}\to \SO{n+2}\to M$) and holonomy irreducible for $n\geq3$ (see e.g.~\cite[Proposition~4.1]{mucha}), we must have $V\in\Ls$ if $n\geq3$, i.e.~ $\Sigma=V_2(\R^k)$. Alternatively, $V\in\Ls$ and
$\Sigma=V_2(\R^k)$ for all $n\geq2$ by Lemma~\ref{t-r} below.

\subsubsection{Real case.}

Assume $\Ls$ has no component in the $1/4$-eigenspace $\Span\{V,JX\}$. Then $\Ls=\R X +\Lv+\Lw\subset\Lp$ where $\Lv$ is a subspace of
\[
(\La\ominus\R X)+\Lp_{\theta_2}=\Span \{ X_{24},\ldots,X_{2,n+2} \} ,
\]
and
$\Lw$ is a subspace of
\[
\R X_{14}   + \Lp_{\theta_1}= \Span \{ X_{14},\ldots,X_{1,n+2} \} .
\]
Since $\dim\Ls\geq2$, either $\Lv\neq0$ or $\Lw\neq0$, and we may assume $\Lw\neq0$ (the other case is similar). Using the action of $\SO{n-1}\subset H$, we can assume that
\[
\Lw = \Span \{ X_{14},\ldots,X_{1,p+2} \}
\]
for some $2\leq p\leq n$.

\begin{lem}
$\Lv\perp J\Lw$.
\end{lem}

\begin{proof} Suppose, to the contrary, that $Z\in\Lv$ has a component
in~$J\Lw$, say $Z=JY+W$, where $Y\in\Lw$ is a unit vector and
$W\perp J\Lw$.
Since
\[ \nabla_ZY^*=\frac12[Y,Z]_{\mathfrak m} = \frac12([Y,JY]_{\mathfrak z}
  +[Y,W]_{\mathfrak z}) = -\frac12 V, \]
$Z$ is a $0$-eigenvector of~$R_X$, and $V$ is an $1/4$-eigenvector
of~$R_X$,  Lemma~\ref{linalg2} implies that $V\in\Ls$, a contradiction
to our assumption in this subsection. \end{proof}

Now $\Lv$ is orthogonal to
\[
J\Lw = \Span \{ X_{24},\ldots,X_{2,p+2} \} ,
\]
so, up to the action of~$H$,
\[
\Lv= \Span \{ X_{2,p+3},\ldots,X_{2,p+2+q} \}
\]
for some $0\leq q\leq n-p$. This shows that $\Ls$ is precisely the tangent space at~$o$ of the totally geodesic embedding of~$S^p(1)\times S^q(1)$ into~$M$, a Riemannian product of spheres of constant curvature~$1$.

\begin{lem}\label{t-r}
Let $\Sigma$ be a totally geodesic submanifold of~$M$ such that $\Ls=T_o\Sigma\subset\Lp$ (i.e.~$\Sigma$ is horizontal at~$o$ with respect to $\pi \colon M\to B$). Then $\Sigma$ is horizontal everywhere and totally real.
\end{lem}

\begin{proof}
Since $\Sigma$ is horizontal at~$o$, $\Ls$ is totally real. We want to show that~$V_p^*\perp T_p\Sigma$ for all $p\in\Sigma$. Writing $p=go$ for $g=\exp (tY)$ and $Y\in\Ls\subset\Lp$, this is equivalent to $\tau_t^{-1}(V^*_p)\perp\Ls$, where $\tau_t$ is parallel displacement along the geodesic $\gamma(t)=\exp(tY)$ from $0$ to $t$. Thanks to equation~(\ref{tau}), \[  \tau_t^{-1}(V^*_p) =  e^{t\nabla Y^*}\cdot d(L_g)_o^{-1}(V_{go}^*)= e^{t\nabla Y^*}\cdot(\mathrm{Ad}_{g^{-1}}V)^*_o= e^{t\nabla Y^*}\cdot(e^{-t\mathrm{ad}_Y}\cdot V)_{\mathfrak m}. \] Since $[Y,V]=JY$, $[Y,JY]_{\mathfrak m}=-V$ and $\nabla Y^*=\frac12[Y,\cdot]_{\mathfrak m}$, this is a vector in the plane $\Span \{ V,JY \} $. But $JY\perp\Ls$, so $\Sigma$ is horizontal at~$p$.
Since $\nabla V^*|_{\mathcal H}=-\frac12J$, this also implies that $\Sigma$ is totally real at~$p$.
\end{proof}

\subsection{The case $X=\frac1{\sqrt2}(X_{13}+X_{24})\in\Ls$}

We may also assume that $\Ls$ does not contain a singular vector of the other type, namely, with isotropy algebra isomorphic to $\mathfrak{so}_{n-1}$, for otherwise we would be in the case of Subsection~\ref{sec:x13}. Note that here $X_+=\frac1{\sqrt2}(X_{23}-X_{14})=JX\in\Lp_{\theta_1+\theta_2}$ and $X_-=\frac1{\sqrt2}(X_{23}+X_{14})\in\Lp_{\theta_1-\theta_2}$, where $\lambda_+=\frac54$, $\lambda_-=0$, and $\Lp_{\theta_1}$ and $\Lp_{\theta_2}$ collapse to the same~$\frac12$-eigenspace of~$R_X$.

\begin{lem}\label{sxm}
$\Ls$ does not contain~$X_-$.
\end{lem}

\begin{proof}
Since $\theta_1(X)=\theta_2(X)$,  we have that $X_-$ centralizes $X$. Every pair of commuting elements of~$\Lp$ spans a subspace which is conjugate to $\La$ under the $K$-action. If $X_-\in\Ls$, then $\Ls$ contains a $K$-conjugate of~$\La$, and thus contains a singular vector of the other type, contrary to our assumption.
\end{proof}

\begin{lem}\label{vxm}
$V\in\Ls$ if and only if $X_+=JX\in\Ls$. In this case $\Ls\cap \Lp$ is a complex subspace of~$\Lp$.
\end{lem}

\begin{proof}
The first statement can be proved as Lemma~\ref{vxpm}.
The second statement follows from Lemma~\ref{3}.
\end{proof}

\begin{lem}\label{36}
If $n\geq3$ and $Y\in\Ls\cap(\Lp_{\theta_1}+\Lp_{\theta_2})$, then $Y=X_{15}+X_{26}$, up to the $H$-action.
\end{lem}

\begin{proof}
Using the action of~$\SO{n-2}$, we may assume that $Y=X_{15}+a X_{25}+ b X_{26}$ with $a^2+b^2\neq0$ and $b>0$ if $n\geq4$ (since $\Ls$ does not contain a singular vector of the other type), and $Y=X_{15}+a X_{25}$ with $a\neq0$ if $n=3$. We have $\sqrt2[X,Y]=-X_{35}-aX_{45}-bX_{46}\in\Lh$. Now we have $\sqrt2R_Y(X)=[Y,[X,Y]]=X_{13}+a\sqrt2X_-+(a^2+b^2)X_{24}\neq0$ and it belongs to $\Ls\cap(\La+\R X_-)$, a contradiction to the fact that $\La+\R X_-$ centralizes $X$ (as in Lemma~\ref{sxm}), 
unless $n\geq4$, $a=0$ and $b=1$. In this case $Y=X_{15}+X_{26}$, as desired.
\end{proof}

Clearly $X_{15}+X_{26}$ is $H$-conjugate to $X_{13}+X_{24}=\sqrt 2 X$, so we 
deduce from Lemmas~\ref{sxm} and~\ref{vxm} that

\begin{lem}\label{37}
Suppose $Y=X_{15}+X_{26}\in\Ls$. Then $\Ls\cap\Span \{ X_{25},X_{16} \}  \subset \Span \{ JY=X_{25}-X_{16} \}$. Moreover, the following are equivalent:
\begin{enumerate}
  \item $JY\in\Ls$,
  \item $V\in\Ls$,
  \item $\Ls\cap \Lp$ is complex.
\end{enumerate}
\end{lem}
\begin{proof}[Proof of Theorem~\ref{tg}]

Proceeding by induction and applying the ideas of Lemmas~\ref{36} and~\ref{37}, we deduce that, up to conjugation either
\[
\Ls=\Span \{ X_{12}, X_{13}+X_{24},X_{23}-X_{14},\ldots, X_{1,2k-1}+X_{2,2k},X_{2,2k-1}-X_{1,2k} \},
\]
or
\[
\Ls=\Span \{ X_{13}+X_{24},\ldots,X_{1,2k-1}+X_{2,2k} \},
\]
for some $k=2,\ldots,[\frac n2]+1$. The corresponding totally geodesic submanifold is the complex Hopf-Berger sphere
\[
\Sigma=\SU{k}/\SU{k-1} \subset \SO{n+2}/\SO{n}=M,
\]
and its maximal totally real submanifold, a round sphere $S^{k-1}$. From~\cite[Subsection~2.4]{OR} and the general fact that the sectional curvature $\sec_g$ for a metric $g$ satisfies $\sec_g = r^2 \sec_{r^2 g}$ for all $r > 0$, we obtain the following equalities for a vertical vector $V$ and horizontal vectors $X$ and $Y$ of the complex Hopf-Berger sphere $\Sigma=S^{k-1}_{\C,\tau}(r)$:
\[
	\tau = \sec_{g_{\tau}}(V, X) = r^2 \sec_{r^2 g_{\tau}}(V, X), \quad
	4 - 3\tau = \sec_{g_{\tau}}(X, Y) = r^2 \sec_{r^2 g_{\tau}}(X, Y).
\]
 The sectional curvatures of the complex Hopf-Berger sphere $\Sigma$ at the $2$-planes $\Span\{V,X\}$ and $\Span\{X,JX\}$ are given by $1/4$ and $5/4$, as follows from Table~\ref{1} for $t=\pi/4$. Hence, $r=\sqrt{2}$ and $\tau=1/2$, and thus $\Sigma$ is isometric to~\smash{$S^{k-1}_{\mathbb C,\frac12}(\sqrt2)$}.
This completes the proof of Theorem~\ref{tg}.\qedhere
\end{proof}

\begin{prop}\label{prop:noho}
	Let $M = V_2(\R^{n+2})$, where $n\geq2$, equipped with a normal homogeneous metric.
	Consider a totally geodesic embedding of~$\R {\mathrm P}^2$ in~$V_2(\R^3)=\R {\mathrm P}^3$, and a totally geodesic embedding of~$V_2(\R^3)$ into~$M$. Then $\R {\mathrm P}^2$ is not extrinsically homogeneous in~$M$. Conversely, any non-extrinsically homogeneous totally geodesic submanifold in~$M$ is congruent to one of these.
\end{prop}
\begin{proof}
	Let us consider the standard embedding of $V_2(\R^3)=\SO3$ into $V_2(\R^{n+2})=\SO{n+2}/\SO n$ such that $\SO3$~is embedded as the upper left diagonal block of~$\SO{n+2}$, and that $\R {\mathrm P}^2$ contains~$o$. Suppose that $\R {\mathrm P}^2$~is the orbit of a subgroup~$L$ of the isometry group of $V_2(\R^{n+2})=\SO{n+2}/\SO n$ through~$o$. We may assume that $L$ is closed and connected, so that $L\subset \SO{n+2}\times \SO2$. Since $L\cdot o \subset V_2(\R^3)$, it is clear that we may also assume that $L$~is a subgroup of $U=\SO3\times \SO2$, where the $\SO3$-factor is the upper diagonal block of~$\SO{n+2}$. Note that the isotropy group $U_o$ is a circle diagonally embedded in $\SO3\times \SO2$. Assume first $U_o\subset L$. Then $\dim L=3$. Consider the projection $\pi_2 \colon L\subset \SO3\times \SO2\to \SO2$ onto the second factor. Then $\pi_2$~is onto, since $U_o$~is diagonally embedded, and $\ker\pi_2$ is a closed and normal $2$-dimensional subgroup of~$L$. This implies that $L$~is a~$3$-torus contained in~$U$, a contradiction. Now assume $U_o\not\subset L$.  Then $\dim L=2$, which implies that $L$~is abelian, so it is a maximal torus of~$U$ and $\R {\mathrm P}^2=L\cdot o$~is diffeomorphic to a finite quotient of~$T^2$; this is a contradiction since these spaces have non-diffeomorphic universal coverings.

	To show the uniqueness, assume that $\Sigma$ is a totally geodesic submanifold of~$M$.
	Then $\Sigma$ is included in at least one maximal totally geodesic submanifold~$\bar \Sigma$ of~$M$ as given by Theorem~\ref{tg}.
	Note that, with the exception of the inclusion $V_2(\R^3) \subset V_2(\R^4)$, it is true for all maximal totally geodesic submanifolds~$\bar \Sigma$ of~$M$ that any isometry in the connected component of the isometry group of~$\bar \Sigma$ extends to an isometry of~$M$, see Lemma~\ref{lemma:stis} and Subsection~\ref{refl}.
	If $\bar \Sigma$ is a product of round spheres, and hence symmetric, then any totally geodesic submanifold of~$\bar \Sigma$ is extrinsically homogeneous.
	If $\bar \Sigma$ is a complex Hopf-Berger sphere, then it follows from the results of~\cite{OR} that any totally geodesic submanifold of~$\bar \Sigma$ is extrinsically homogeneous.
	It remains the case where $\bar \Sigma$ is a totally geodesic Stiefel manifold~$V_k(\R^{n+2})$ of lower dimension.
	Provided that $n\ge2$, we can use the same argument recursively,  and we will either end up with an extrinsically homogeneous totally geodesic submanifold, or $\Sigma$ is as described in the first part of the proof.
\end{proof}
%%%%%%%%%%%%%%%%%%%%%%%%%%%%%%%%%%%%%%%%%%%%%%%%%%%%%%%%%%%%%%%%%%%%%%%%%%%%%%%%%%%%%%%%%%%%%%%%%%%%%%%%%%%%%%%%%%
\section{Polar actions on the Stiefel manifold of real orthonormal two-frames}
%%%%%%%%%%%%%%%%%%%%%%%%%%%%%%%%%%%%%%%%%%%%%%%%%%%%%%%%%%%%%%%%%%%%%%%%%%%%%%%%%%%%%%%%%%%%%%%%%%%%%%%%%%%%%%%%%%

\begin{lem}\label{lm:toru}
  Let $L \subseteq \SO{n+2}\times\SO2=\mathrm{Isom}(V_2(\R^{n+2}))^0$ be a closed connected subgroup acting isometrically and with orbits of dimension at least two on~$V_2(\R^{n+2})$.
Then $L$ is conjugate to a subgroup of~$\SO{n+2} \times \SO2$ such that the tangent space of the $L$-orbit through~$o$ contains a vertical vector.
\end{lem}

\begin{proof}
Let $\pi_1 \colon \SO{n+2}\times\SO2 \to \SO2$ be the projection onto the second factor. Let $L'$ be the kernel of~$\pi_1|L$. Then $L'$ is a closed subgroup 
of~$\SO{n+2}$ of positive dimension. After conjugation, we may assume that the maximal torus of~$L'$ consists of matrices of the form ($r=[\frac n2]$)
\[
\begin{pmatrix}
 \rotmat{t_1}  &  &  \\
  & \ddots &  \\
  &  & \rotmat {t_r} \\
\end{pmatrix},
\]
if $n$ is even and of the form
\[
\begin{pmatrix}
 \rotmat{t_1}  &  &  &  \\
  & \ddots &  &  \\
  &  & \rotmat {t_r} &  \\
  &  &  &  \onemat \\
\end{pmatrix},
\]
if $n$ is odd and where in both cases $t_1, \dots, t_r \in \R$.
Since the dimension of~$L'$ is positive, we may assume, possibly after conjugation of~$L'$ in~$\SO{n+2}$, that there are elements in~$L'$ as above where $t_1$ runs through all real numbers. It follows that there is a one-parameter subgroup of~$L'$ such that the tangent space to the orbit through~$o$ spans~$\Lz$.
\end{proof}

\begin{proof}[Proof of Theorem~\ref{pa}]
If $M=V_2(\R^{n+2})$ is equipped with an $\SO{n+2}$-invariant metric, then $\Lm=\Lz + \Lp$ is a decomposition into $\SO n$-invariant subspaces, where $\SO n$ acts through two copies of its standard representation on~$\Lp$ and trivially on~$\Lz$. It follows that $\Lz \perp \Lp$ for all $\SO{n+2}$-invariant metrics on~$M$.

Note first that a polar action on~$V_2(\R^{n+2})$ cannot have one-dimensional principal orbits, since otherwise the sections would be totally geodesic hypersurfaces, which contradicts the main result of~\cite{tsukada}.
Thus Lemma~\ref{lm:toru} implies that any polar action on~$M = V_2(\R^{n+2})$, $n \ge2$ has a section which is horizontal at one point and hence everywhere by~Lemma~\ref{t-r}. Therefore, the action is as described 
as in~\cite{mucha} and its orbits are preimages of orbits of a polar action 
on the corresponding Grassmannian~$G_2(\R^{n+2})$.
In particular, the action is hyperpolar and the cohomogeneity is less than or equal to~$2$.
\end{proof}

%%%%%%%%%%%%%%%%%%%%%%%%%%%%%%%%%%%%%%%%%%%%%%%%%%%%%%%%%%%%%%%%%%%%%%%%%%%%%%%%%%%%%%%%%%%%%%%%%%%%%%%%%%%%%%%%%%
\section{Homogeneous hypersurfaces of the\\ real, complex and quaternionic Stiefel manifolds}
%%%%%%%%%%%%%%%%%%%%%%%%%%%%%%%%%%%%%%%%%%%%%%%%%%%%%%%%%%%%%%%%%%%%%%%%%%%%%%%%%%%%%%%%%%%%%%%%%%%%%%%%%%%%%%%%%%

In this last section we will consider all of the real, complex and quaternionic Stiefel manifolds
\begin{equation}\label{eq:stmf}
V_k(\R^n) = \frac{\SO n}{\SO{n-k}}, \quad
V_k(\C^n) = \frac{\SU n}{\SU{n-k}}, \quad
V_k(\H^n) = \frac{\Sp n}{\Sp{n-k}}.
\end{equation}
whose elements are $k$-tuples of orthonormal vectors in~$\R^n$, $\C^n$, or~$\H^n$.
We will determine the homogeneous hypersurfaces of these manifolds, or, in other words,
we will classify all isometric actions of compact 
connected Lie groups of cohomogeneity one, up to orbit equivalence,
and prove Theorem~\ref{thm:hhls}.

\begin{rmk}
\label{rmk:irreducible}
The irreducibility of the metric induced by~(\ref{eq:metric}) on $M=\SO{n+2}/\SO{n}$ is probably well-known but we include a proof for completeness. First notice that if $n\neq 2$, the Lie group $\SO{n+2}$ is simple, thus $M$ is irreducible by \cite[Ch.~X Cor.~5.4]{kn}.

Suppose $n=2$. Then, the subspace $\g{m}$ is spanned by $\{X_{12}, X_{13}, X_{14}, X_{23}, X_{24}\}$. 
Now recall that $R$ can be regarded as a linear endomorphism of $\Lambda^2(\g{m})$. Using equations~\eqref{r-h} and~\eqref{r-hv}, we compute
\begin{equation}
\label{eq:computationcurvature}
\begin{aligned}
R(X_{12}\wedge X_{13})&=\frac{1}{4} X_{12}\wedge X_{13},& R(X_{12}\wedge X_{14})&=\frac{1}{4} X_{12}\wedge X_{14},\\
R(X_{12}\wedge X_{23})&=\frac{1}{4} X_{12}\wedge X_{23},& R(X_{12}\wedge X_{24})&=\frac{1}{4} X_{12}\wedge X_{24}.
\end{aligned}
\end{equation}
It is clear that these operators generate a Lie algebra isomorphic to~$\g{so}_5$, so it follows from the Ambrose-Singer Theorem that $M$ is an irreducible Riemannian manifold (with generic holonomy).
\end{rmk}

To find all isometric cohomogeneity one actions on a Stiefel manifold, we can assume that the action is given by a closed connected subgroup of the isometry group.  Therefore, it is sufficient to study connected subgroups of the connected component of the isometry group of the Stiefel manifold.

\begin{lem}\label{lemma:stis}
The connected component of the isometry group of a Stiefel manifold $V_k(\F^n)$, $\F = \R, \C, \H$, $1 \neq k \neq n-1$,  equipped with a normal homogeneous metric, is of the form $G \cdot K_1$ (almost direct product), where
\begin{align*}
G&=\SO n,& K_1 &= \SO{k} & \mbox{if}\;\; \F &= \R, \\
G&=\SU n,& K_1 &= \SU{k}\cdot\U1 & \mbox{if}\;\; \F &= \C, \\
G&=\Sp n,& K_1 &= \Sp{k} & \mbox{if}\;\;\F &= \H,
\end{align*}
where the action is given by~(\ref{eq:stac}).
\end{lem}

\begin{proof}
Note that the spaces in the statement of the lemma are irreducible by Remark~\ref{rmk:irreducible} and \cite[Ch.~X Cor.~5.4]{kn}.
Now the statement follows from~\cite[Corollary~1.3]{reggiani}.
\end{proof}

From now on we will assume without loss of generality that isometric Lie group actions on a Stiefel manifold are given by closed subgroups of~$G \times K_1$ (direct product). The following result allows us to derive the classification of cohomogeneity one actions on Stiefel manifolds from the known classifications of cohomogeneity one~\cite{kollross} and transitive actions~\cite{oniscik} on Grassmannians.

\begin{lem}\label{proj}
Let $M = V_k(\F^n)$, $\F = \R, \C, \H$ be a Stiefel manifold and let $B = G_k(\F^n) = G/K$ be the corresponding Grassmannian, where $G = \SO n$, $\SU n$, 
$\Sp n$. Suppose there is an isometric Lie group action of cohomogeneity one on~$M$ given by a closed subgroup~$L$ of~$G \times K_1$. Let $\pi_1 \colon G \times K_1 \to G$ be the projection onto the first factor. Then the following statements hold:
\begin{enumerate}
	\item The subgroup  $\pi_1(L)$ acts isometrically on~$B$ and this action is either of cohomogeneity one or transitive.
	\item  $\pi_1(L) \neq G$.
\end{enumerate}
\end{lem}

\begin{proof}
Note that $L$~is contained in the closed subgroup~$\pi_1(L) \times K_1$ of~$G \times K_1$.
Thus the action of~$L$ on~$M$ is of the same or higher cohomogeneity as the action of~$\pi_1(L) \times K_1$ on~$M$, whose cohomogeneity is equal to that of the action of~$\pi_1(L)$ on~$B$.

For the second part of the statement, we observe that if~$\pi_1(L) = G$, then it follows that $L$ is of the form $G \times Q$, where $Q \subseteq K_1$ is some closed subgroup.
Such a group acts transitively on~$M$, a contradiction.
\end{proof}

In the sequel we assume that $2 \le k \le n-2$ since homogeneous hypersurfaces of spheres are well understood, see~\cite{HL}.
Also, note that $G_k(\F^n)$ is isometric to~$G_{n-k}(\F^n)$, where the isometry is given by mapping a $k$-dimensional linear subspace to its orthogonal complement in~$\F^n$.

\begin{lem}\label{prop:cost}
Let $(L,B)$ be a pair as in Table~\ref{t:pairs}. Then the group $L \times K_1$ acts with cohomogeneity one on the corresponding Stiefel manifold~$V_k(\F^n)$.
Conversely, all homogeneous hypersurfaces of the Stiefel manifold~$V_k(\F^n)$, $k=2, \dots, n-1$, which arise as preimages of homogeneous hypersurfaces of the Grassmannian~$G_k(\F^n)$ or~$G_{n-k}(\F^n)$ under the projection~$\pi$ are given by the pairs $(L,B)$ in
Table~\ref{t:pairs}.
\begin{table}[h]
			\begin{tabular}{|l|l|}
\hline
            $L$ & $B$ \\ 
\hline\hline
			$\SO{n-1}$ & $G_k(\R^n)$ \\
\hline
			$\SUxU{n-1}{1}$ & $G_k(\C^n)$ \\
\hline
			$\Sp{n-1}\times\Sp{1}$ & $G_k(\H^n)$  \\
\hline
			$\U{m}$ & $G_2(\R^{2m})$  \\
\hline
			$\U{m}$ & $G_3(\R^{2m})$  \\
\hline
			\end{tabular}\qquad
			\begin{tabular}{|l|l|}
\hline
            $L$ & $B$ \\ 
\hline\hline
			$\Sp{m}$ & $G_2(\C^{2m})$  \\
\hline
			$\Sp{m}$ & $G_3(\C^{2m})$ \\
\hline
			$\G$ & $G_3(\R^7)$  \\
\hline
			$\Spin9$ & $G_2(\R^{16})$ \\
\hline
			$\Sp{m}\Sp1$ & $G_2(\R^{4m})$ \\
\hline
			\end{tabular}
	\caption{ \str Pairs~$(L,B)$ for which the subgroup~$L$ of~$G$ acts with cohomogeneity one on~$B=G_k(\F^n)$.}\label{t:pairs}
\end{table}
\end{lem}
\begin{proof} 
If a connected subgroup $L$ acts with cohomogeneity one on~$B=G_k(\F^n)$, then $L\times K_1$ acts with cohomogeneity one on $M=V_k(\F^n)$.
We may extract the classification for cohomogeneity one actions on Grassmannians from~\cite[Theorem B]{kollross}, thus yielding Table~\ref{t:pairs}.

Conversely, for $2 \le k \le \frac{n}2$, any closed connected subgroup of~$G$ acting with cohomogeneity one on~$B=G_k(\F^n)$ is after conjugation contained in the subgroup~$L$ of~$G$ from one of the pairs~$(L,B)$ in Table~\ref{t:pairs}.
\end{proof}
	
\begin{lem}\label{rmk:trgr}
Using the notation of Lemma~\ref{lemma:stis}, for all of the following pairs~$(L,B)$, the subgroup~$L$ of~$G$ acts transitively on~$G_k(\F^n)$:
\begin{align*}
(\Spin7, G_2(\R^8)), \;
(\Spin7, G_3(\R^8)), \;
(\G , G_2(\R^7)).
\end{align*}
Conversely, for $2 \le k \le \frac{n}2$, and $n\neq4$ in case~$\F=\R$, if a closed connected proper subgroup of~$G$ acts transitively on~$B=G_k(\F^n)$, then the pair~$(L,B)$ is one of the three given above.
\end{lem}

\begin{proof}
See \cite[Table~7]{oniscik} for the classification of transitive actions on Grassmannians.
\end{proof}

\begin{proof}[Proof of Theorem~\ref{thm:hhls}]
Let $L$ be a closed connected subgroup of~$G \times K_1$, where $G$ and~$K_1$ are as in Lemma~\ref{lemma:stis}. Assume $L$ acts with cohomogeneity one on~$M$.
Let $\pi_1 \colon G \times K_1 \to G$ be the projection on the first factor.
Then Lemma~\ref{proj} shows that the group~$\pi_1(L)$ either acts with cohomogeneity one on~$B=G_k(\F^n)$ or transitively; and furthermore, that~$\pi_1(L) \neq G$. 
The case where the action on~$B$ is of cohomogeneity one is treated in Lemma~\ref{prop:cost}.
We may therefore now assume that the group $\pi_1(L)$ acts transitively on~$B$.
It now only remains to decide whether the actions on Grassmannians given 
in Lemma~\ref{rmk:trgr} give rise to cohomogeneity one actions on the corresponding Stiefel manifolds. 
However, it has been proved in~\cite{oniscik} that the actions of~$\Spin7$ on~$V_2(\R^8)$, of~$\Spin7$ on~$V_3(\R^8)$ and of~$\G$ on~$V_2(\R^7)$ are transitive.
It follows by dimension counting that the actions of~$\Spin7$ on~$V_6(\R^8)$, of~$\Spin7$ on~$V_5(\R^8)$ and of~$\G$ on~$V_5(\R^7)$ are of cohomogeneity greater than one.
This completes the proof of the theorem. 
\end{proof}

%%%%%%%%%%%%%%%%%%%%%%%%%%%%%%%%%%%%%%%%%%%%%%%%%%%%%%%%%%%%%%%%%%%%%%%%%%%%%%%%%%%%%%%%%%%%%%%%%%%%%%%%%%%%%%%%%%
%%%%%%%%%%%%%%%%%%%%%%%%%%%%%%%%%%%%%%%%%%%%%%%%%%%%%%%%%%%%%%%%%%%%%%%%%%%%%%%%%%%%%%%%%%%%%%%%%%%%%%%%%%%%%%%%%%

%%%%%%%%%%%%%%%%%%%%%%%%%%%%%%%%%%%%%%%%%%%%%%%%%%%%%%%%%%%%%%%%%%%%%%%%%%%%%%%%%%%%%%%%%%%%%%%%%%%%%%%%%%%%%%%%%%
\bibliography{ref}
\bibliographystyle{amsplain}
%%%%%%%%%%%%%%%%%%%%%%%%%%%%%%%%%%%%%%%%%%%%%%%%%%%%%%%%%%%%%%%%%%%%%%%%%%%%%%%%%%%%%%%%%%%%%%%%%%%%%%%%%%%%%%%%%%

\end{document}